\documentclass[12pt]{article}

\usepackage{amsmath,amsthm,amsfonts,latexsym,amscd}

\usepackage[mathscr]{eucal}

\theoremstyle{plain}
            \newtheorem{theorem}{Theorem}[section]
            \newtheorem{lemma}[theorem]{Lemma}
            \newtheorem{corollary}[theorem]{Corollary}
            \newtheorem{proposition}[theorem]{Proposition}
             
\theoremstyle{definition}

            \newtheorem{notation}[theorem]{Notation}

\theoremstyle{remark} 
\newtheorem{remark}[theorem]{Remark}

\numberwithin{equation}{section}

\addtocounter{section}{0}

\newcommand{\script}[1]{{\mathcal{#1}}}

\newcommand{\spp}{\script P}

\newcommand{\gl}{\lambda}
\newcommand{\gb}{\beta}
\newcommand{\ga}{\alpha}

\newcommand{\gs}{\sigma}

\newcommand{\bthrm}{\begin{theorem}}
\newcommand{\ethrm}{\end{theorem}} 
\newcommand{\blem}{\begin{lemma}}
\newcommand{\elem}{\end{lemma}}
\newcommand{\bprf}{\begin{proof}}
\newcommand{\eprf}{\end{proof}}

\newcommand{\bnn}{\gb\mathbb N\setminus \mathbb N}

\newcommand{\scp}{\script P}

\newcommand{\ds}{\displaystyle}

\newcommand{\iab}{I_{\text{ab}}}

\begin{document}

\title{The Kadison-Singer problem for the direct sum of matrix algebras}
\author{Charles Akemann, Joel Anderson and Betul Tanbay}

\maketitle

\markright{\large The Kadison-Singer Problem}

\begin{abstract}
Let $M_n$ denote the algebra of complex $n\times n $ matrices and write $M$ for the direct sum of the $M_n$.  So a typical element of $M$ has the form
\[
x = x_1\oplus x_2 \cdots \oplus x_n \oplus \cdots,
\]
	where $x_n \in M_n$ and $\|x\| = \sup_n\|x_n\|$.  We set  $D= \{\{x_n\} \in M: x_n$ is diagonal for all $N\}$.  We conjecture (contra Kadison and Singer (1959)) that every pure state of $D$ extends uniquely to a pure state of $M$.  This is known for the normal pure states of D, and we show that this is true for a (weak*) open, dense subset of all the singular pure states of $D$.  We also show that (assuming the Continuum hypothesis) $M$ has pure states that are not multiplicative on any maximal abelian *-subalgebra of $M$.

\end{abstract}

\bigskip

Let $M_n$ denote the algebra of complex $n\times n $ matrices and write $M$ for the direct sum of the $M_n$.  So a typical element of $M$ has the form
\[
x = x_1\oplus x_2 \cdots \oplus x_n \oplus \cdots,
\]
where $x_n \in M_n$ and $\|x\| = \sup_n\|x_n\|$. We embed $M_n$ in $M$ in the natural way and $M$ in $B(H)$ (the algebra of all bounded operators on a Hilbert space) in the natural way by viewing elements of $M$ as block diagonal matrices.  We use $D$ to denote the diagonal elements in $M$ ($D$ also comprises the diagonal elements of $B(H)$ in this embedding).  We write 
$\spp$ for the unique conditional expectation of $B(H)$ (or $M$) onto $D$. Our purpose here is  to study the  Kadison-Singer problem in this context.  Our conjecture  in the context of $M$ is:

\medskip

  CONJECTURE: Each pure state of $D$ has a unique state extension to $M$.
  
 \medskip
 
 \noindent
 This problem has many equivalent formulations. In their original paper Kadison and Singer  ``inclined to the view'' that  our conjecture had a negative answer.  We take the opposite position here

In what  follows, we could take the
sizes of the matrix summands of $M$ to be any finite
numbers that tend to infinity
with $n$.  Our choice has been made purely for notational convenience. It is well known that all of the pure states on $D$ extend uniquely  to $M$  if and only if  they all extend uniquely to $B(H)$.    In certain ways  $M$ is simpler than
$B(H)$.  For example $M$ does not contain a non-atomic MASA (i.e. maximal abelian *subalgebra). In other ways  $M$ is more complicated than $B(H)$. For example it has a large center and a complicated  ideal structure. 
$M$ is also an appealing venue for study of this problem since it may be possible to avoid the set-theoretic issues that arise  in $B(H)$.

Let's address the last point
first.   We  begin by introducing some notation.   Recall that an element $b \in B(H)$ (or $b \in M)$ is called {\em paveable}  if for all $\epsilon > 0$ there exist an integer $ m > 1$ and projections $p_1, ..., p_m \in D$  with $\ds  \sum_1^m p_j = 1$ such that 
\[
\big \|\sum_1^m p_j(b-\spp(b))p_j \big \| < \epsilon  .
\]
  There are two natural approaches  to this problem.   A well known result of the second author \cite{ext}  states that our conjecture is true iff each element of $M$ (or $B(H)$) is paveable.  One
can either show that more and more elements of $B(H)$ or $M$ are paveable, or one can show
that more and more singular pure states of $D$ have
unique extension to all of $M$ or $B(H)$.  In  $B(H)$  only the first has
been successful without the use of the continuum hypothesis or similar set theoretic assumptions.  In particular, {\bf no }singular pure states of $D$ have been shown to have unique state extensions to $B(H)$ without  using  the continuum hypothesis or something similar.

As we show in Section 3, the situation is very different for $M$.  In $M$ there is a weak*
relatively open, unitarily invariant, dense subset of the singular pure states of $D$ for which pure
state extensions are unique.  In particular, the set of singular pure states of $D$ that fail to have unique state extensions to $M$ (if any) is of first category  in the set of all singular pure states of $D$ with the weak* topology.  At present, even using additional set theoretic assumptions, there is no result this strong for $B(H)$.  We shall also show that  every element of the ideal $I_{ab}$ of $M$ generated by the abelian projections is paveable.

In the final section, we show that, assuming the Continuum Hypothesis, there are pure states of $M$ that are not pure on any 
MASA of $M$.

\section{Ideals in $M$}

	Let us begin by presenting some facts about  ideals in $M$.  {\bf In this paper the term ideal will always refer to a norm-closed two-sided ideal.}  The ideal structure of $M$ is quite rich and our discussion is not complete.  Some of the results presented here are known.    Nevertheless it seems worthwhile to collect them and to point out connections to the Kadison-Singer problem.   Let us now introduce some more notation.   Write $\tau_n$ for the normalized trace on $M_n$ and 
\[
\|a\|_2 = (\tau_n(a^*a))^{1/2}.
\]

Ideals in $M$  were studied by Wright in \cite{wr}.  He  called a subset $L$ of projections in $M$ a {\em $p$-ideal} if:
\begin{enumerate}
\item it is closed under equivalence and 
\item it is closed under the formation of sups and infs
\end{enumerate}  
and he proved the following theorem  in \cite{wr}. We write $Z(M)$ for the center of $M$.

 \begin{theorem}  The following statements hold.

\begin{enumerate}

\item   The $p$-ideals in the set of projections in $M$ are in one-to-one correspondence with the ideals in $M$.

\item Each ideal in $M$ is generated by its projections.

\item  
There is a one-to-one correspondence between the maximal ideals of $M$ and the maximal ideals in  $Z(M)$.
\end{enumerate}

\end{theorem}

The center  of $M$ consists of elements of the form $\oplus _{n=1}^\infty \gl_nI_n$, where  $\{\gl_n\}$ is a bounded complex  sequence 
and $I_n$ is the identity matrix in $M_n$.  In other words the center of $M$ can be identified with $\ell^\infty$ (which may also be  identified with $C(\gb\mathbb N)$, where $\gb\mathbb N$ denotes the Stone-Cech compactification of the natural numbers.

The first conclusion of the following Lemma is based on an idea of Sorin Popa.  Details of the proof were provided by David Sherman. The second conclusion is based on the proof of Theorem 1 of \cite{jp}.

\begin{lemma}  If $x\in M_n$  with $\spp(x)=0$ (i.e. $x$ has zero diagonal) and $G$ denotes the group of symmetries (i.e. self-adjoint unitary operators) in $D$, then

\begin{enumerate}
 
 \item There is $d\in G$ such that  
 \[
 \|xd-dx\|_2 \geq \sqrt2 \|x\|_2. 
 \]
 
 \item If $x=x^*$, there is $u \in G$ such that 
  \[
 \|xu-ux\| \geq \|x\|. 
 \]
 \end{enumerate}
\end{lemma}

\begin{proof}
Direct calculation (using $\spp(x)=0$) shows that for any
diagonal element $d$  with diagonal entries $(d_1, ..., d_n) \in G,$
\[
\|xd-dx\|_2^2 =  (1/n) \sum_{i \ne j} |x_{ij}|^2  |d_j - d_i|^2 .
\]

Choose Haar measure $\mu$ on $G$  to have total mass 1.  Because each $d_i$  takes the
values $+1, -1$ with equal probability and independently, $|d_j - d_i|^2 = 4$ or $0$ with equal  probability
for $i \ne j$,  so $i \ne j$,
 $\int_G |d_j - d_i|^2
 d\mu(d) = 2.$  Thus

$$\int_G 
\|xd-dx\|_2^2 d\mu(d) =
\|xd-dx\|_2^2 =  (1/n) \sum_{i \ne j} |x_{ij}|^2\int_G |d_j - d_i|^2
 d\mu(d) $$
     $$=   (1/n) \sum_{i\ne j} 2 |x_{ij}|^2 d\mu =  2 \|x\|_2^2.$$
Since $\mu(G)=1$, for some  $d \in G, \|xd-dx]\|_2^2 \ge2 \|x\|_2^2.$

For the  second assertion we may assume that $\|x\| =1$.  Since  $ux-xu$ has zero diagonal for any $u \in G, u(ux-xu)=x-uxu$ also has zero diagonal.   Thus $\int_G (x-uxu)\ d\mu(u)$  also has zero diagonal, as does $uxu$.  Next observe that
 $ \int_G uxu\,d\mu(u) $ is a diagonal matrix (by direct calculation) .  Since  $ \int_G uxu\,d\mu(u) $ also has zero diagonal (as above), it is 0.  This means that 
 \[
 \int_G u(ux-xu)\ d\mu(u)=\int_G x \ d\mu(u) - \int_G uxu \ d\mu(u) = x.
 \]
 Thus $1=\|x\| \le \int_G \|ux-xu\| \  d\mu(u)$, so (since $\mu(G)=1$) $\|ux-xu\| \ge 1$, for some  $u \in G$.

\bigskip

\end{proof}

\bigskip

\begin{notation}Since  the center $Z$ of $M$ may be identified with $\ell^\infty = C(\gb \mathbb N)$, a closed ideal $\Gamma$ in the center is determined a closed  subset  of $\gb\mathbb N$.
Let $\Omega$ denote such a closed subset and define two ideals in $M$ as follows.  Write
\begin{align*}
I_2(\Omega) &= \{x= \oplus x_n:  \lim_{\omega} \|x_n\|_2 = 0, \omega \in \Omega\} \text{ and }\\
 I_\infty(\Omega) &= \{ x= \oplus x_n :\lim_{\omega} \|x_n\|  =0,\omega\in \Omega\}.
\end{align*}
It is clear that $I_2(\Omega)  \cap Z(M)= I_\infty(\Omega)\cap Z(M)  = \Gamma$.
\end{notation}
\bigskip

\begin{proposition} Incorporating the notation above, $I_2(\Omega)$ is the maximal ideal $J$ such that $J\cap Z(M)=\Gamma$  and $I_\infty(\Omega)$ is the minimal such ideal.  
\end{proposition}

\begin{proof}  If $x \in M$, it is clear that $||x_n||_2 \le ||x_n||\ \  \forall n$, hence $I_\infty(\Omega) \subset I_2(\Omega)$.   Now fix an ideal $J$ with 
\[
J\cap Z(M) = \Gamma.
\]
Since $\Gamma$ contains an approximate unit for   $I_\infty(\Omega)$, $J $ must  contain $I_\infty(\Omega)$.

   For the maximality claim, by contradiction suppose that there is a self-adjoint element  $a \in J$  and $\omega \in \Omega$ such that $\ds\lim_\omega\|a_n\|_2 \ne 0$.  Using spectral theory, we may replace $a$ by a projection $p$.  Since the 2-norm is continuous on $\gb\mathbb N$, there are a central projection $r \in M$ and a natural number $k$  such that $\tau_n(rp) \ge (1/k)\tau_n(r)\ \forall n$, and $\ds  \lim_{\omega}\|r_n\|_2 =1$. Since $J$ is an ideal, there are partial isometries $v_1, ... v_k \in M$ such that $\sum_1^k v_j^*pv_j \ge r$, hence $r \in J$.  This means that $J \cap Z(M) $ is strictly larger than $\Gamma$ which gives  the desired  contradiction.
\end{proof}

Notation:  For any ideal $J$ of $M$,  $x_J$ denotes  the support projection
of $J$ in $M^{**}$,   $y_J$ denotes the supremum in $M^{**}$ of projections in
$D\cap J$, and $z_J$ denotes the support projection in $M^{**}$ of $J\cap Z(M)$. Clearly $x_J \ge y_J \ge z_J$.  Also, if $a$ is a self--adjoint element of $M$ and  $\gs$ is a measurable subset of $\mathbb R$, we write $\chi_\gs(a)$ for the spectral projection of $a$ determined by $\gs$.

\begin{proposition} The following statements hold for an ideal $J$ of $M$.
\begin{enumerate}
\item  $a \in (J)_{sa}$ if and only if $a \in M_{sa}$ and  $1-\chi_{(-1/n,
1/n)}(a) \in J$ for all $n$.


\item  $J\cap D$ is a MASA in $J$.

\item   $y_J$ is dense in $x_J$.

\item   $y_J$ is regular in $J^{**}$.

\item   $\scp(J)\subset J\cap D$.

\item   If $J=I_2(\Omega)  \text{ or }   I_\infty(\Omega), \,\,\,  (D+J)/J$ is a MASA in $M/J$.

\end{enumerate}
\end{proposition}

\begin{proof}
Fix $a = a^*$ in $J$.   Since ideals are
hereditary and 
\[
na \ge 1-\chi_{(-1/n, 1/n)}(a),
\]
it follows that   $1-\chi_{(-1/n, 1/n)}(a)\in J$.

Next suppose $a$ is a self adjoint element of $ M$ and  $p_n = 1-\chi_{(-1/n, 1/n)}(a)$ is in  $J$ for all $n$.   Since $J$ is norm closed, $ap_n$ is in $J$  and
$\| a- ap_n\| \le 1/n$, it follows that $a \in J$.  So the assertion in part (1) is true.


Now fix a self-adjoint element $a$ in $J$ that commutes with elements of $J\cap D$. Since $J$ contains the operators of finite rank,
$a$ must be diagonal, so $a\in J \cap D$ and therefore assertion (2) is true.

Next if the assertion in (3)  is false, then there is a self
adjoint $a \in J$ that is orthogonal to
all the projections in $D \cap J$. 
This contradicts part (2).

(4) follows immediately from \cite {regularity} Lemma 4.1(iii), because $y_J$ dominates the central cover of $J \cap I_{\beta(N)\setminus N}$.

For (5), let $p \in J$ be a projection.  Fix an integer $k > 0$, and write 
\[q = \chi_{(1/k,1]}(\spp(p)). 
\]
  Then it  suffices to show that $q \in J$ since $||q\spp(p) - \spp(p)|| < 1/k$ and $J$ is norm closed.  Now observe that

\[
\text{rank}(p_n)=n\tau_n(p_n) =n \tau_n(\spp(p)_n)\ge (n/k)\tau_n(q_n)\ge (1/k) \text{rank}(q_n)\tag{$*$}
\]
and define a projection $r \in M$ as follows.  If rank$(p_n)\ge n/k$, then set $r_n=1_n$.  If  rank$(p_n)< n/k$, then 
choose  unitaries $\{v_j\}_1^k \in M$ such that for each $n, 
r_n:=\sum_{j=1}^k v_jp_nv_j^*$ is a projection; necessarily rank$(r_n)$= $k$ rank$(p_n)$.

We claim that $r \in J$.  Let $s\in M$ be the central projection defined by $s_n=r_n$ if $r_n=1_n$, and $s_n=0$ if not.  Clearly $s \in J$ since there exist unitaries $w_1, ..., w_k$ in $M$ such that $\sum_1^k w_npw_n^* \ge s$. Thus $sr \in J$ and $(1-s)r \in J$ by its construction as a finite sum of projections in $J$. This establishes our claim.

 Now note that   $\text{rank}(r_n) \ge \text{rank}(q_n)$ for all $n$ by $(*)$ above.  Thus there is a unitary $U \in M$ such that $UrU^* \ge q$, and hence $q \in J$.  Since $k$ was arbitrary and $J$ is closed, it follows that we also have  $P(p) \in J$, as desired.

For (6),  suppose first that $ J = I_\infty(\Omega)$, then fix   $x$ in $M$ such that $x$ has  0 diagonal and $xd-dx =0$ in the quotient $M/J$ for all $d \in D$.  Thus $\ds\lim_\omega\|dx-xd\|_\infty=0 $ for all  $\omega\in \Omega$ and for all $d \in D$. By Lemma 1.2, $\ds\lim_{\omega}\|x\|_\infty=0$ for every $\omega$ in $\Omega$, i.e. $x \in I_\infty(\Omega) $.  This gives the result.

If $J=I_2(\Omega)$, then fix an element  $x$ in  $M$ such that $x$ has 0 diagonal and $xd-dx\in J$ for all $d \in D$.  Thus $\ds\lim_{\omega}\|xd-dx\|=0$ for each  $\omega$ in  $\Omega$ and for all $d \in D$. By Lemma 1.2, $\lim_{\omega}\|x\|_2 =0$ for each  $\omega$ in $ \Omega$, so  $x$ is in $I_2(\Omega)$.

\end{proof}

\noindent {\bf Remarks.}

(a)  We know by Theorem 1.5.7 and Corollary 1.5.8 in \cite{gert}, that for any ideal $J$ of $M$, the algebra $J+D$ is norm closed and that its quotient $(J+D)/J$ is a C*-subalgebra of $M/J$. It follows that  if the quotient   $(D+J)/J$ is not a MASA of $M/J$, then our Kadison - Singer  Conjecture would be  false.  For this reason we take assertion (6) of the last theorem as more evidence for our conjecture.

(b) If an ideal $J$ has the form  $I_\infty$, then $y_J=x_J=z_J$.

\bigskip
\section{Some General C*-Algebra Results}
\bigskip

\bigskip

Recall that a MASA is a maximal abelian *subalgebra of a $C^*$- algebra, and, if $f$ is a state on 
a $C^*$- algebra $A$, then its {\em hereditary kernel}  is by definition
\[
K_f = \{x \in A : f(xx^*+x^*x)=0\}.
\]  
Recall that as defined in \cite{excising} a net $\{a_\ga\}$ of positive norm one elements in a $C^*$-algebra $A$ {\bf excises} a state $f$  of $A$ if 
\[
\lim_\ga \|a_\ga a a_\ga -f(a)a_\ga^2\| = 0  \qquad \text{for every $a$ in $A$}.
\]
We are interested in pure states, so we recall the basic facts. First by  Proposition 2.2 of \cite{excising}, if $f$ is a pure state of a C$^*$-algebra $A$, then its hereditary kernel is a maximal hereditary C$^*$-subalgebra. Further, the covering projection $p$ of this hereditary kernel is the limit in $A^{**}$ of its approximate unit, and $1-p$ (i.e. the support projection of $f$) has rank 1  (\cite{gert}, 3.13).

\bigskip

\begin{theorem}  If $f$ is a pure state of the C\,$^*$-subalgebra $B$ of $C$
 and $\{b_\alpha\}$ is any excising net for $f$ in $B$, then $f$
has unique state extension to $C$ if and only if $\{b_\alpha\}$ converges to a rank one
projection in $C^{**}$.
\end{theorem}

\begin{proof} Suppose $\{b_\ga\}$ converges to a projection $p$ of rank one in the double dual of $C$ and suppose
$g$  and $h$ are  states on $C$ that extend $f$.   If we fix $c \in C$ we have 
\[
\lim_\ga g(c) = \lim_\ga g(b_\ga cb_\ga) = g(pcp)= g(\gl p)= \gl = \lim_\ga h(b_\ga c b_\ga) = h(c)
\]

For the converse suppose that $f$ has a unique state extension $g$ to all of $C$. It follows that $\{b_\ga\}$ must converge to a projection of rank one in $C^{**}$, else the extension would not be unique.

\end{proof}

\begin{proposition} Let $A$ be a unital C*-algebra and $f$ a pure state on $A$.  There there is a MASA  of $A$ on which $f$ is pure and from which $f$ has unique state extension if and only if  the hereditary kernel of $f$ has an abelian approximate unit.

\end{proposition}

\begin{proof} Suppose that the hereditary kernel of $f$ has an abelian approximate unit $\{b_\alpha\}$  and let $B$ be any MASA of $A$ that contains $\{b_\alpha\}$. 
Write  $p$ in $A^{**}$ for  the limit of the $b_\alpha$'s.  Since $f $ is pure,  $p$ has
codimension 1 as  mentioned in the introduction to section 2.  Since $p$ also lies
in $B^{**}$, it has codimension 1 there also.  Since $f $ is supported on $1-p$ which  has  dimension 1 in $ B^{**}$, its restriction to $B$ is pure and it must have a unique
state extension.

Now suppose that  there there is a MASA $B$ of $A$ on which $f$ is pure and from which $f$ has unique state extension.  Then the approximate unit of $\{x \in B : f(xx^*+x^*x)=0\}$ must converge in $A^{**}$ to a projection of codimension 1, else state extensions of $f$ would not be unique.

\end{proof}

\bigskip

\begin{lemma}Let $A$ be a C*-algebra and $x$ be an open, dense (i.e. $\overline{x}=1)$,
central projection in $A^{**}$.  Let $S(A)$ (resp., $PS(A)$) denote the states (resp.,pure states) of $A$
and $S(A)_x$ (resp., $PS(A)_x$) denote those states (resp., pure states) that take the value 1 on $x$.
Then $S(A)_x$  (resp. $PS(A)_x$) is weak* dense in $S(A)$ (resp. $PS(A)$).
\end{lemma}

\begin{proof}
Let $a_\alpha\uparrow x, a_\alpha \in A$.  The  complement of $S(A)_x$ (resp. $PS(A)_x$) in
 $S(A)$ (resp. $PS(A)$) consists of those $f$ such that $f(a_\alpha)=0$ for all $\alpha$,
and that is clearly closed.

That $S(A)=\overline{S(A)_x}$ follows since the central projection $x$ is regular.  The
conclusion $PS(A)=\overline{PS(A)_x}$ follows from $S(A)=\overline{S(A)_x}$ by
\cite[Appendix B14] {Dixmier}.

\end{proof}

\bigskip

 \section{The Ideal $I_{ab}$ Generated by the Abelian Projections}
\bigskip

Let $\iab$ denote the ideal generated by the abelian projections in $M$.  It is clear that a projection $p=\oplus p_n\in M$ is abelian if and only if  the rank of $p_n$ is 0 or 1 for each $n$.
In some sense, the existence the ideal $\iab$ is what makes $M$ so different from $B(H)$.  We define the projection $p_B \in I_{ab}$  by the condition that $(p_B)_n$ has $1/n$ in each of its entries.  Recall that 
\[
I_2(\gb \mathbb N\setminus \mathbb N) = \{ \oplus x_n: \lim_\omega \|x_n\|_2 = 0, \omega \in \bnn\},
\]
and let $K$ denote the compact operators in $M$, i.e. $K=I_\infty( \gb\mathbb N \setminus \mathbb N).$

\begin{proposition} The following statements hold.
\begin{enumerate}
\item $I_{ab} \subset I_2( \gb\mathbb N \setminus \mathbb N)$.

 \item Any ideal $J$ that is strictly larger than $K$ contains an infinite dimensional abelian projection.
 
 \item $K \neq I_{ab} \neq I_2( \gb\mathbb N \setminus \mathbb N)$.
 
 \end{enumerate}
\end{proposition}
\begin{proof}
For the first assertion, it suffices to show that every abelian projection lies in $I_{ \gb\mathbb N \setminus \mathbb N}$.  This is clear since for any abelian projection $p= \oplus p_n \in M$, 
where   $\| p_n\|_2 \le  1/\sqrt{n}$.

For the second assertion,  suppose the ideal $J$ contains a non-compact operator.
It follows that $J$ contains a projection of infinite rank which, in turn,  must dominate an infinite dimensional abelian projection.

For the third assertion, if $p$ is an abelian projection such that rank($p_n)=1$ for all $n$, then $p \in I_{ab}
\setminus K$. 

 If $q$ is a projection in $M$ such that rank($q_{2^n}) = n$ and rank($q_k)=0$ when $k$ is not a power of 2, then $q \in I_2( \gb\mathbb N \setminus \mathbb N)\setminus I_{ab}$ as follows.  
 $$\|q_{2^n}\|_2= \sqrt n/2^{n/2} \rightarrow_{n\to \infty} 0.$$  Thus $q \in  I_2( \gb\mathbb N \setminus \mathbb N)$. If $b$ is a self-adjoint  linear combination of $t$ abelian projections, then  for all $n$,
rank($b_n) \le t$.  Since rank($q_n) \to \infty, ||(b-q)_n|| \ge 1$ when rank$(q_n) > t$.  Thus $q$ can't be in $I_{ab}$.

\end{proof}

\begin{proposition}  If $\omega$ is a free ultrafilter  and $q= \oplus q_n$ is a   projection in 
$ I_2(\{\omega\}) \cap D$,  then 
\[
 \lim_{\omega}  \|(p_Bq)_n\|= 0.
\]
\end{proposition}  
\begin{proof} Since $\ds\lim_\omega \|q_n\|_2 = 0$, if $\epsilon > 0$ we may select a set $\gs \in \omega$ such that  $\|q_n\|_2 < \epsilon$ for each  $n$  in $\gs$.
Since $q \in D$,   $(qp_bq)_n$ can be viewed as a block matrix in $M_n$ for each $n \in \sigma$.  The size of the block is rank($q_n)\times$ rank($q_n)$ and the entries in the block are all $1/n$.  I.e. this block is a scalar multiple of a rank 1 projection, and the multiple is rank($q_n)/n\le\epsilon^2$.  Thus \[
 \lim_{\omega}  \|(p_Bq)_n\|= 0.
\]
\end{proof}

\begin{corollary} If $p$ is a projection in $D\cap I_2( \gb\mathbb N \setminus \mathbb N)$, then $pp_B$ is compact.
 \end{corollary}
 
\begin{proof}
 Apply  Proposition 3.2 to each $\omega \in ( \gb\mathbb N \setminus \mathbb N).  $
\end{proof}

\begin{corollary} For any proper ideal $J $ that contains $ p_B$, $y_J\ne x_J$.
 \end{corollary}
\begin{proof}
 If $y_J= x_J$, then the projection approximate unit for $D \cap J$, which converges to $y_J$, would be an approximate unit for $p_B$, and that is false by the last proposition.
\end{proof}

\begin{proposition}  If $f$ is a pure state of $M$, $p$ is an abelian projection in $M$ such that 
$f(p) = 1$, and $B$ is a MASA of $M$ that contains $p$, then   $f|_B $   is pure and $f$ is its unique state extension.
\end{proposition}

\begin{proof} Since every  MASA of $M$ is unitarily equivalent to $D$, we may assume  $B = D$.

Let $f'$ be any state extension of $f|_D$ and recall that  $f'=pf'p$ by Schwarz inequality.  Now observe that $pMp \subset D$ because $p$ is abelian.   So, for any $a \in M$,
\[
f'(a)=f'(pap)=f(pap)=f(\scp(pap))=f(p\scp (a)p)=f(\scp(a)).
\]
Thus $f$ is the unique state extension of $f|_D$.

If  $f|_D$ is  not pure,  there would be a projection $q \in D$ such that $0 < f(q) < 1$  and this would mean that for any $a \in M$, 
\begin{align*}
f(a) &= f(\scp(qaq+qa(1-q)+(1-q)aq+(1-q)a(1-q)))\\
&=f(\scp(qaq+(1-q)a(1-q))\\
&=f(qaq+(1-q)a(1-q)) \\
&= ((qfq)+(1-q)f(1-q))(a),
\end{align*}
 so $f=qfq+(1-q)f(1-q)$, contradicting the assumption that $f$ is pure on $M$.  So $f$ is a pure state on $D$.
\end{proof}

\bigskip

\begin{theorem} The set of all singular pure states of $M$ that take a nonzero value  on $I_{ab}$ is relatively weak* open and dense in the set of all singular pure states of $M$. 
\end{theorem}

\begin{proof}The central cover $x$ of the ideal $I_{ab}/K$  in $(M/K)^{**}$ is an open projection.  If $x$ were not dense for $M/K$, there would be a  non-zero hereditary subalgebra $J_0$ of $M/K$ that is orthogonal to  $I_{ab}/K$. Lift $J_0$ to a closed hereditary subalgebra $J$  of $M$. Since $J$ is generated by its projections,  it contains an infinite
rank projection $p$. Thus by the form of $M$ there is an infinite rank abelian projection q in $I_{ab}$ such
that pq=q. However this contradicts the assumption that $J$  is orthogonal to $I_{ab} /K$. Thus $x$ is open and dense.
Apply Lemma 2.3 to $M/K$ to get that $PS(M/K)x$ is weak* dense in $PS(M/K)$.  Since $x$ is central and open, $PS(M/K)_x = \{f \in PS(M/K) : f(x) > 0\}$, $PS(M/K)_x$ must be weak* open in $PS(M/K)$.
\end{proof}

\begin{theorem} The set of all singular pure states of $D$ that take a nonzero value  on $I_{ab}\cap D$ is relatively weak*
open and dense in the set of all singular pure states of $D$. All of these pure states of $D$ have unique state extension to all of $M$.
\end{theorem}
\begin{proof}
The first sentence is proved in the same way as the last theorem.  The second sentence follows from Theorem 3.5.
\end{proof}



\begin{theorem} Every element of $I_{ab}$ can be paved.
\end{theorem}
\begin{proof}
Let $\epsilon > 0 $ be given, let $q$ be an abelian projection in $M$ and choose $m > 2/\epsilon$.  By Proposition 1.2 of \cite{t2}, it will suffice to prove that $q$ is paveable. The rank of each $q_n$ is 1 or 0, so there are abelian projections $\{p_\alpha\}_{\alpha = 1}^m$ such that if $p$ is a rank 1 projection in $D$ such that $pp_\alpha = 0$ for all $\alpha = 1, ..., m$, then $||pqp|| < \epsilon/2$.  Set $r = 1-\sum_{\alpha = 1}^m p_\alpha$.  Since each $p_{\alpha}$ is abelian and in $D, p_\alpha(q-\spp(q))p_\alpha = 0$.  

Define the map $\Psi : rD \to qMq$ by $ \Psi(b) = qbq$.  Since $q$ is abelian, the range of $\Psi$ intersects each $M_n$ in either $\{0\}$ or the 1-dimensional space of multiples of $q_n$. By the definition of $r$, for each $n$ and each abelian projection $s \in rD, \|(sqs)_n\|=\|(qsq)_n\|=\|(\Psi(s))_n\|_1 \le 
\epsilon/2$.  Choose $k$ as a power of 2 such that  $1/k < \epsilon$.  Thus by Theorem 5.7 of \cite{Ly} (applied inside each $M_n$), there are projections $\{p_\alpha\}_{\alpha = m+1}^{m+k} \in rD$ such that $\sum_{\alpha =m+1}^{m+k} p_\alpha=r$ and, for $\alpha = 1, ..., k$,
$||\Psi(p_{m+\alpha})-(1/k)qrq|| \le \epsilon.$ Thus$\||\Psi(p_{m+\alpha})\| =\|qp_{m+\alpha}q\| =\|p_{m+\alpha}qp_{m+\alpha}\|< 2\epsilon$.  Thus
$||p_{\alpha}qp_{\alpha}|| \le  2\epsilon$ for all $\alpha = 1, ..., m+k$.  I.e. $q$ is paveable.
\end{proof}

\bigskip

\begin{theorem}  If $J$ is any ideal of $M$ that is not contained in $I_{ab}$, then there is a projection $p \in J\cap D$ and a $*$-isomorphism of $pMp$ onto $M$.
\end{theorem}

\begin{proof} Since $J$ is not contained in $I_{ab}$, there is a projection $q \in J \cap D$ such that $\text{rank}(q_n)\to \infty$. 
Thus by a routine recursive construction, there is a subsequence $\{n_k\}$ and projections $p_{n_k} \in D_{n_k}$ such that $\text{rank} (p_{n_k}) = k$ and $p_{n_k} \le q_{n_k}$ for each k = 1, 2, 3, ...    Set $p=\sum_k p_{n_k}$.  Since $p \le q$, then $p \in J\cap D$.  Since $p_{n_k}Mp_{n_k}$ is *isomorphic to $M_k$, there is a *isomorphism of $pMp$ onto $M$. 
\end{proof}

\begin{theorem}
If $J$ is any ideal of $M$ that is not contained in $I_{ab}$, and if our   Kadison Singer Conjecture is false, then there is a pure state of $D\cap J$ that does not have unique state extension to $J$.
\end{theorem}

\begin{proof}  By assumption there is a pure state $h$ of $D$ that does not have unique state extension to $M$.  
By the last theorem there is a *isomorphism $\psi$  of $pMp$ onto $M$ for some projection $p \in J\cap D$.  Note that $p \in D$ implies that $\psi|_{pDp}$ is a *isomorphism onto $D$. If $f_1, f_2$ are distinct state extensions of $h$ to $M$, then $\psi^*(f_1), \psi^*(f_2)$ are distinct states of $pMp$.  If $a \in D$, then there exists unique $b \in D$ such that $\psi^{-1}(b) = pap$.

Now define $\theta : M \to pMp$ by $\theta (c) = pcp$.  Remark that $\theta^*$ is a positive isometry from $(pMp)^*$ onto $pM^*p$.  Thus the states $g_i=\theta^* \circ\psi^*(f_i)$ for i=1,2, are also distinct states.  Compute:
$$g_i(a)= \theta^* \circ\psi^*(f_i)(a)= \psi^*(f_i)
(pap)=\psi^*(f_i)(\psi^{-1}(b))=f_i(b)=h(b)=$$
$$h(\psi^{-1}( pap)=\psi^{-1*}(h)(\theta(a))=\theta^* \circ\psi^{-1*}(h)(a)$$
  Thus $g_1, g_2$ are distinct states of $M$ that live on $J$ and have the same pure restriction $\theta^* \circ\psi^{-1*}(h)$ to $D$.
\end{proof}

\begin{theorem}
If $J$ is any ideal of $M$ that is not contained in $I_{ab}$, and if our   Kadison Singer Conjecture is false, then there is an element of $J$ that cannot be paved.
\end{theorem}

\begin{proof} 
As in the proof of the last theorem, there is a *isomorphism $\psi$  of $pMp$ onto $M$ for some projection $p \in J\cap D$.  Note that $p \in D$ implies that $\psi|_{pDp}$ is a *isomorphism onto $D$.  Thus if our conjecture is false, there is an element of $M$ that can't be paved, so the image in $J$ of that element under $\psi^{-1}$ can't be paved.
\end{proof}

We have already seen (Theorem 3.8) that every element of $I_{ab}$ is paveable. 
We now compare $I_{ab}$ with a large C*-subalgebra of $M$ that has the same property.
We will say that 
$b \in M$ is an {\em almost permutation} if each $b_n$
has at most one non-zero entry in each row and column, that entry being $1$. 
The C*-algebra $\cal A$ generated by $D$ and the almost permutations
will be called the {\em permutation algebra}.   
It was shown in \cite{l1} that every element of $\cal A$ is paveable.  

We shall now give a 
characterization of the elements of $\cal A$.
Say that $b \in M$ is $d$-{\it empty} if there is a natural number $d$  
such that each $b_n$ has at most $d$ non-zero entries in each row and column. 
Let $\cal F$ be the *subalgebra of operators that are d-empty for some 
$d \in \mathbb N$. It was shown \cite {aver} that an operator $b$ is in $\cal F$ 
iff there are a $m \in \mathbb N$, $c_1,...,c_m \in D$ and near permutations $b_1,...,b_m$ such 
that $b = \sum_{i=1}^n c_ib_i$; thus $\cal A$ is the norm-closure of $\cal F$.
$\cal A$ also contains each element $b \in M$ whose non-zero entries have modulus larger than 
some positive number $\delta$, since $b_n$ is $({\delta \over\|b_n\|})^2$-empty.

\begin{remark}
${\cal A} + I_{ab}$ is a C*- subalgebra of $M$ by Corollary 1.5.8 of \cite{gert}, and every element in this subalgebra is paveable.  We don't know of a larger C*-subalgebra of $M$ with this property.  It is known (\cite {aver}) that neither ${\cal A} \subset I_{ab}$ nor $I_{ab} \subset {\cal A} $ holds.  
The question left hanging in \cite{aver} is whether $p_B$ in $\cal A$.  The answer, which surprised us, is ``yes".
Our intuition was that if $b$ is any $d$-empty matrix, and $n$ is much larger than $d$, then subtracting $b$ from $p_B$ affects only a ``small" 
fraction of the entries of $(p_B)_n$, and hence the resulting norm can't be all that small.
A graph theoretic technique shows that this intuition is wrong.  We thank Michel de la Salle and 
Gilles Pisier for pointing us to  a  theorem from graph theory (Theorem 1.1 of \cite{JF}),which is due to Friedman and which is exactly what is needed here.
\end{remark}

\begin{theorem}
$p_b \in \cal A$.
\end{theorem}
\begin{proof}
The range of $(p_B)_n$ is the 1-dimensional subspace spanned by $(1, ..., 1)= v$.  If $b$ is any permutation matrix in $M_n$, then $bv=v$.  This suggests that if we take an even integer d (much smaller than $n$), we might approximate $(p_B)_n$ by  a  $d$-empty element $s \in \cal A$ of the form 
$$s_n= (1/d)\sum_{\alpha=1}^{d/2} (b_\alpha+b_\alpha^*),$$
where each $b_\alpha$ is a permutation matrix in $M_n$.

To estimate $\|s_n-(p_B)_n\|$ we note that $(s_n-(p_B)_n)(v)=0$, so it suffices  to carry out our estimate on the orthogonal complement of $v$.  Since $(p_b)_n$ vanishes there, we are actually concerned with estimating the modulus of the second largest eigenvalue of the self adjoint operator $s_n$.

At this point Friedman's paper \cite{JF} suggests that we should show that $s_n$ is the type of matrix that he is dealing with in his Theorem 1.1.  That is,  we need to show that $s_n$ is the adjacency matrix of a d-regular undirected graph on $n$ vertices (times $1/d$).  If the matrix $b_\alpha$ corresponds to the permutation $\pi_\alpha$ of $1, ..., n$. such a graph $G_n$ is formed by defining the edges
(\cite{JF} p. 2) as follows:
$$E = \{(i,\pi_\alpha(i)), (i, \pi_\alpha^{-1}(i)) : \alpha = 1, ..., d/2; i = 1, ..., n\}.$$
The adjacency matrix of $G$ is exactly $ds_n$.

Now we can apply Theorem 1.1 of \cite{JF} to conclude that there are positive constants $\eta, r$, {\bf independent of $n$}, such that with probability at least $1-\eta/n^r$, the second eigenvalue of a random matrix of the same form as $s_n$ will be less than $(2\sqrt{d-1} + 1)/d$.  Of course that expression can be made as small as we like by taking a large $d$, and this is independent of $n$.

 If $\delta > 0$ is given, we choose even $d$ so that 
$(2\sqrt{d-1} + 1)/d < \delta$.  Let $n_0$ be large enough so that $1-\eta/n_0^r > 0$.  Since $\cal A$ clearly contains all of the finite rank operators in $M$, we may define an operator $c\in \cal A$ by letting $c_n = (p_B)_n, n < n_0$.  For $n \ge n_0$ we choose $c_n$ to one of those elements in $M_n$ of the form $s_n$ (for the value of $d$ chosen above) whose second eigenvalue is less than $\delta$.  These choices ensure that $c \in \cal A$ and that $||p_B - c|| < \delta$. 

\end{proof}

This proof is a classic use of the probabilistic method in combinatorics.  One shows the existence of something by showing that a random choice has a positive probability of being a right choice.  

We call a projection $q \in M$ a {\bf Hadamard projection} if rank$(q_n)$ is 0 or 1 for each $n$, and if rank($q_n)=1$, then $q_n$ has a eigenvector of the form $(e^{i\theta_1}, ..., e^{i\theta_n})$, for real $\theta_1, ..., \theta_n$.

\begin{corollary}
If $q$ is a Hadamard projection in $M$, then $q \in \cal A$.
\end{corollary}
\begin{proof}
Using the notation above, let $w$ be a unitary in $D$ such that 

\noindent $w_n(e^{i\theta_1}, ..., e^{i\theta_n})=(1, ..., 1)$.  Then $w^*p_Bw =q \in \cal A$.  
\end{proof}

\bigskip
 \section{Pure States  of M}

\begin{proposition}  If $f_0$ is a pure state of $D$ and if $f$ is defined on $M$ by $f=f_o\circ P$, then $f$ is a 
pure state of $M$.
\end{proposition}

\begin{proof} 
The proof of \cite[Corollary 3]{pos} applies without change. 
\end{proof}

\begin{theorem} If

\begin{enumerate}
\item   $\script N$ is a type $II_1$ von Neumann algebra factor that has the cardinality of the continuum,

\item $\tau$ denotes the faithful normal tracial state on $\script N$,

\item  $C(\script N)$ is a  collection of MASAs of $\script N$ that has the cardinality of the continuum, and 
\item  the  continuum hypothesis holds,

\end{enumerate}

\noindent then there is a pure state of $\script N$ that is not multiplicative when restricted to any MASA in $C(\script N)$.
\end{theorem}

\begin{proof} 
The proof is essentially given in Lemma 0.5-Theorem 0.7 of \cite{face}.  However, those results assume that $\script N$ acts on a separable Hilbert space.  Thus we need to be sure that each time that hypothesis is used in \cite{face}, we can show here that it is not needed under the hypotheses of the present theorem.
 
 1.  To modify the proof of Lemma 0.5 of \cite{face} to handle the present theorem, we note that in a $II_1$ factor, any two projections with the same trace are unitarily equivalent.
 
 2.  To modify the proof of Lemma 0.6 of \cite{face} to handle the present theorem, we need only show that for any projection $p \in \script N$ and any singular state $f$ on $\script N$ such that $f(p)=1$, there is a sequence $\{p_n\}$ of projections in $\script N$ such that $f(p_n)=1$ for all $n$ and $p_n \downarrow 0$ strongly.
 Let $q_j$ be a maximal orthogonal family of projections under $p$ for which $f(q_j)=0$.  Since $\tau$ is faithful, this family must be at most countable, and by \cite{tak2} the set of $\{q_j\}$ must be infinite.  Set $p_j=\sum_{i=j}^\infty q_i$.  This works since $\tau(p_j) \to 0$ because $\tau$ is normal.
 
 3. To modify the proof of Theorem 0.7 of \cite{face} to handle the present theorem, we need only note that both $\script N$ and $C(\script N)$ have cardinality $c = \aleph_1$ under CH.
 
 \end{proof}
 
 \begin{corollary} 
 There is a pure state $g$ of $M$ that is not multiplicative on any MASA of $M$.
 \end{corollary}
\begin{proof}
If $J$ is any maximal ideal of $M$ and $Q_J$ is the quotient map of $M$ onto $M/J$, then by \cite{feld}, $Q_J(M)$ is a $II_1$ factor that has the cardinality of the continuum and  the set of images of the MASAs of $M$ also has the cardinality of the continuum. Thus the  Corollary follows by taking the pure state $f$ on $Q_J(M)$ that is non-multiplicative  on each of the images of the MASAs of $M$ from the last theorem and defining $g=f\circ Q_J$.
\end{proof}


\bigskip








\newpage


\begin{thebibliography}{aaaaaaaa}


\bibitem{Ly}
Akemann, Charles A., Anderson, Joel, {Lyapnov Theorems for Operator Algebras}, {\it Memoirs Amer. Math. Soc. \bf 458},(1991)


\bibitem{excising}
Akemann, Charles A., Anderson, Joel and Pedersen, Gert K., {Excising states of
C*-algebras}, {\it Canad.\ J.\ Math. \bf 38}(5),(1986),
223-230.


\bibitem{regularity}
Akemann, Charles A., Eilers, Soren,
{Regularity of Projections Revisited}, {\it J. Operator Theory \bf 48}, 2002, 515-534.


\bibitem{face}
Akemann, Charles A. and Weaver, Nik, {Classically normal pure states}, {\it Positivity \bf 11},(2007),
617-625.


\bibitem{pos} Anderson, Joel. Extreme points in sets of positive maps on $B(H)$, {\it  J. Functional Analysis},  {\bf 31} (1979), 195-217.


\bibitem{ext} Anderson, Joel, Extensions, Restrictions and Representations of States on $C^*$-algebras, 
{\it Trans. Amer. Math. Soc.}, {\bf{249}} (1979), 303-329.


\bibitem{j}
Anderson, Joel, {Pathology in the Calkin algebra},  in {\it Topics in Modern
Operator Theory}, Birkha\"user,(1981), pp.\ 27-43.


\bibitem{Dixmier}
Dixmier, J., 
{\it C*-algebras}, North Holland,
London, 1977.


\bibitem{feld}  Feldman,J.
{Nonseparability of certain finite factors},  {\em
Proc. Amer. Math. Soc.
\textbf {7}} (1956), 23-26.




\bibitem{JF}  J. Friedman, {\em  A Proof of Alon's Second Eigenvalue Conjecture and Related Problems}, Mem. Amer. Soc. {\bf 195}, (2008).


\bibitem{jp}  Johnson, B. E. and Parrott, S. K.,
{Operators commuting with a von Neumann algebra modulo the set of compact operators},  {\em
J. Functional Analysis
\textbf {11}} (1972), 39-61.


\bibitem{kr}
Kadison, Richard.\ V.\ and Ringrose, John, 
{\it Fundamentals of the theory of operator algebras, vol. 2}, Academic Press,
London, 1986.


\bibitem{ks}
Kadison, Richard.\ V.\ and Singer, I.\ M., { Extensions of pure states},
{\it Amer.\ J.\ Math.\ \bf 81},\ (1959), 383-400.


\bibitem{gert} Pedersen,\ Gert K., {\em C*-Algebras and their
automorphism groups}, Academic Press, London, 1979.


\bibitem{tak2} Takesaki, Masamichi, {Singularity of positive functionals} {\em
Proc. Japan Acad.
\textbf {35}} (1959), 365-366.


\bibitem{tak} Takesaki, Masamichi, {\em Theory of Operator Algebras I},
Springer-Verlag, New York, 1979.


\bibitem{t2} Tanbay, Betul, {\em Extensions of pure states on algebras of operators}, Dissertation, University of California at Berkeley, 1989.


\bibitem{l1} Tanbay,Betul, {\em Pure State Extensions and Compressibility of the $l_1$-algebra}, Proc. Amer. Soc. {\bf 113}, (1991), 707-713. 


\bibitem{aver} Tanbay, Betul,  {\em Approximating the Averaging Operator}, Tr. J. of Math. {\bf 16}, (1992), 85-94.


\bibitem{wr} Wright, Fred B., {\em A Reduction for Algebras of Finite Type}, Ann. Math., Second Series, \textbf{60}, No. 3,  (1954) , 560-570.


\end{thebibliography}
\end{document}